\documentclass[12pt]{amsart}
\usepackage{amsfonts,amssymb,amscd,amstext,mathrsfs,color, comment}

\input xy
\xyoption{all}

\theoremstyle{theorem}
\newtheorem{theorem}{Theorem} [section]
\newtheorem{proposition}[theorem]{Proposition}

\newtheorem{maintheorem}{Theorem}

\theoremstyle{definition}
\newtheorem{definition}[theorem]{Definition}
\newtheorem{example}[theorem]{Example}

\newcommand{\Q}{{\mathscr Q}}
\newcommand{\A}{{\mathscr A}}
\newcommand{\C}{{\mathbb C}}
\newcommand{\CC}{\mathscr C}
\newcommand{\G}{{\mathscr G}}
\renewcommand{\H}{{\mathscr H}}

\renewcommand{\O}{{\mathscr O}}
\newcommand{\F}{{\mathscr F}}
\newcommand{\LF}{{\mathscr {LF}}}
\newcommand{\R}{{\mathbb R}}
\newcommand{\Z}{{\mathbb Z}}
\newcommand{\Aut}{{\operatorname{Aut}}}

\newcommand{\Der}{{\operatorname{Der}}}
\newcommand{\codim}{{\operatorname{codim}\,}}

\newcommand{\inv}{{^{-1}}}
\newcommand{\git}{/\!\!/}
\newcommand{\pr}{{\operatorname{pr}}}
\newcommand{\red}{{\operatorname{red}}}
\renewcommand{\phi}{\varphi}
\newcommand{\GL}{\operatorname{GL}}
\newcommand{\SL}{\operatorname{SL}}
\newcommand{\SO}{\operatorname{SO}}
\newcommand{\ind}{\operatorname{ind}}
\newcommand{\res}{\operatorname{res}}

\newcommand{\ci}{C^\infty}
\newcommand{\gf}{{\operatorname{fin}}}

\newcommand{\Diff}{{\operatorname{Diff}}}
\renewcommand{\int}{\operatorname{int}}
\newcommand{\Iso}{{\operatorname{Iso}}}
\newcommand{\ql}{{\operatorname{ql}}}

\begin{document}

\title{Equivariant Oka Theory: \\ Survey of Recent Progress}

\author{Frank Kutzschebauch, Finnur L\'arusson, Gerald W.~Schwarz}

\address{Frank Kutzschebauch, Institute of Mathematics, University of Bern, Sidlerstrasse 5, CH-3012 Bern, Switzerland}
\email{frank.kutzschebauch@math.unibe.ch}

\address{Finnur L\'arusson, School of Mathematical Sciences, University of Adelaide, Adelaide SA 5005, Australia}
\email{finnur.larusson@adelaide.edu.au}

\address{Gerald W.~Schwarz, Department of Mathematics, Brandeis University, Waltham MA 02454-9110, USA}
\email{schwarz@brandeis.edu}


\thanks{F.~Kutzschebauch was supported by Schweizerischer Nationalfonds grant 200021--116165.}

\thanks{The authors thank Volodya Ezhov and Gerd Schmalz for the invitation to write this survey}

\date{5 December 2021}


\begin{abstract}
We survey recent work, published since 2015, on equivariant Oka theory.  The main results described in the survey are as follows.  Homotopy principles for equivariant isomorphisms of Stein manifolds on which a reductive complex Lie group $G$ acts.  Applications to the linearisation problem.  A parametric Oka principle for sections of a bundle $E$ of homogeneous spaces for a group bundle $\G$, all over a reduced Stein space $X$ with compatible  actions of a reductive complex group on $E$, $\G$, and $X$.  Application to the classification of generalised principal bundles with a group action.  Finally, an equivariant version of Gromov's Oka principle based on a new notion of a $G$-manifold being $G$-Oka.
\end{abstract}

\maketitle

\tableofcontents

\section{Introduction} 
\label{sec:intro}

\noindent
This is a survey of recent work, published since 2015, on equivariant Oka theory, mainly from our papers \cite{KLS-2015}, \cite{KLS-2017}, \cite{KLS-2017-lin}, \cite{KLS-2018}, and \cite{KLS-2021}.  Oka theory is the subfield of complex geometry that deals with various homotopy principles, in this context collectively known as the Oka principle, stating that the obstructions to solving certain analytic problems on Stein spaces are purely topological.  The work surveyed here incorporates group actions -- holomorphic actions of complex Lie groups -- into such homotopy principles.  This work can be also be viewed as part of holomorphic geometric invariant theory.

Oka theory has its roots in the pioneering work of Kyoshi Oka.  The Oka principle first appeared in his 1939 result that a holomorphic line bundle on a Stein manifold is trivial if it is topologically trivial.  Oka theory was developed much further in the late 1950s to early 1970s, starting with Grauert's foundational papers \cite{Grauert-Approximation}, \cite{Grauert-Liesche}, and \cite{Grauert-Faserungen}.  Other key contributors in this period were Cartan \cite{Cartan} and Forster and Ramspott \cite{Forster-Ramspott}.  The focus was on complex Lie groups and, more generally, complex homogeneous spaces, a typical result being that every continuous map from a Stein space to a complex homogeneous space can be deformed to a holomorphic map.  In a seminal paper of 1989 \cite{Gromov}, Gromov initiated the modern development of Oka theory.  He discovered a way to generalise the results of the Grauert period beyond complex homogeneous spaces to a larger class of manifolds that he named elliptic.  They possess a geometric structure called a dominating spray that mimics the exponential map of a Lie group and makes it possible to solve various analytic problems by linearising them.  Over the past 20 years, modern Oka theory has been vigorously developed.  In particular, the optimal weakening of ellipticity for Oka principles to hold was identified and the class of Oka manifolds defined in \cite{Larusson-2004} and \cite{Forstneric-2009}.  Forstneri\v c's monograph \cite{Forstneric-book} is a comprehensive up-to-date reference on Oka theory.

A very brief review of the geometric invariant theory relevant here starts with the foundational 1973 paper of Luna \cite{Luna-1973}.  He studied the action of a reductive complex algebraic group $G$ on an affine variety $X$ and proved his famous slice theorem.  A consequence is that the quotient variety $X\git G$ has a natural stratification, and if $X$ is smooth, the map of $X$ to $X\git G$ is a $G$-fibre bundle over each stratum.

The holomorphic version of the theory, for a reductive complex Lie group acting on a Stein space, was developed by Snow \cite{Snow} and Heinzner \cite{Heinzner-1988}, \cite{Heinzner-1991}.  A reductive complex Lie group is automatically algebraic, so there is a strong connection to the algebraic theory.  Geometric invariant theory and Oka theory of the Grauert period were first brought together in the 1995 paper of Heinzner and Kutzschebauch \cite{Heinzner-Kutzschebauch}.  The work surveyed here builds on and continues their work and, in our most recent paper, brings Gromov's Oka principle into geometric invariant theory.

Most of the work surveyed here can be summarised in six main results, Theorems A--F below.  The following sections provide further details, relevant definitions, brief sketches of proofs, and other related results.  In the final section, we list some open problems.

Let a complex reductive group $G$ act holomorphically on a Stein manifold $X$.  The categorical quotient $X\git G$ is a normal Stein space that parametrises the closed $G$-orbits in $X$.  Let $\pi:X\to X\git G$ be the quotient map (we sometimes write $\pi_X$) and for $Z\subset X\git G$ let $X_Z$ denote $\pi\inv(Z)$. If $Z=\{q\}$ is a point, we write $X_q$ instead of $X_{\{q\}}$.  The $G$-finite holomorphic functions  (Section  \ref{sec:isomorphisms}) on  $X_q$ give $X_q$ an algebraic structure which may be neither reduced nor irreducible.
The pullback by $\pi$ of the sheaf of holomorphic functions on $X\git G$ is the sheaf of $G$-invariant holomorphic functions on $X$.  The quotient has a locally finite stratification by locally closed smooth subvarieties, called the Luna stratification, such that points $q, q'\in X\git G$ lie in the same stratum if and only if the fibres $X_q$ and $X_{q'}$ are $G$-biholomorphic (equivalently, the algebraic structures on $X_q$ and $X_{q'}$ are equivariantly algebraically isomorphic).  If $S$ is a stratum of $X\git G$, then $\pi\inv(S)\to S$ is a holomorphic $G$-fibre bundle (whose fibre need not be smooth).

In Sections \ref{sec:isomorphisms} and \ref{sec:linearisation} we consider the following problem. Let $X$ and $Y$ be Stein $G$-manifolds. If $X$ and $Y$ are $G$-biholomorphic, then $X\git G$ and $Y\git G$ are biholomorphic, preserving the stratifications. So let us assume that $X$ and $Y$ have a common stratified quotient
$Q\simeq X\git G\simeq Y\git G$.  To have an Oka problem, let us also assume that there is an  open cover $(U_i)$    of $Q$ and $G$-biholomorphisms $\Psi_i\colon X_{U_i}\to Y_{U_i}$ inducing the identity on $U_i$.  We say that $X$ and $Y$ are  \emph{locally $G$-biholomorphic over $Q$}. Now our equivariant Oka problem is to see what kind of continuous or smooth $G$-isomorphism of $X$ and $Y$ implies that there is a $G$-biholomorphism.

\begin{maintheorem}  \label{mainthm:isomorphisms1}
Let $G$ be a reductive complex Lie group.  Let $X$ and $Y$ be Stein $G$-manifolds locally $G$-biholomorphic over a common quotient.

{\rm (a)}  Any strict $G$-diffeomorphism $X\to Y$ is homotopic, through strict $G$-diffeo\-morphisms, to a $G$-biholomorphism.

{\rm (b)}  Any strong $G$-homeo\-morphism $X\to Y$ is homotopic, through strong $G$-homeo\-morphisms, to a $G$-biholomorphism.
\end{maintheorem}

A $G$-diffeomorphism $X\to Y$ inducing the identity map of the quotient $Q$ is called \emph{strict} if it induces a biholomorphism between $X_q$ and $Y_q$, with their reduced structures, for all $q\in Q$.  The definition of a strong $G$-homeomorphism is somewhat involved and will be given in Section \ref{sec:isomorphisms}.  Roughly speaking, a strong $G$-homeomorphism restricts to a $G$-biholomorphism $X_q\to Y_q$ for each $q\in Q$ that depends continuously on $q$.  A strict $G$-diffeomorphism is not necessarily a strong $G$-homeomorphism
\cite[Example 3.2]{KLS-2017}.

It turns out that one can often deduce that $X$ and $Y$ are locally $G$-biholomorphic over $Q$ from the existence of strict or strong $G$-isomorphisms! See Section \ref{sec:isomorphisms} for the definitions  of ``infinitesimal lifting property'' and ``large'' used in Theorems \ref{mainthm:isomorphisms2} and \ref{mainthm:linearisation}.

\begin{maintheorem}  \label{mainthm:isomorphisms2} Let $G$ be a reductive complex Lie group.  Let $X$ and $Y$ be Stein $G$-manifolds  with common quotient $Q$. If there is a strict $G$-diffeomorphism or a strong $G$-homeomorphism $X\to Y$ and $X$ has the infinitesimal lifting property, then $X$ and $Y$ are locally $G$-biholomorphic over $Q$, hence $G$-biholomorphic.
\end{maintheorem}

The so-called linearisation problem has a long history (see \cite{Huckleberry} and \cite{Kraft-1996}).  It asks whether the action of a reductive complex group $G$ on affine space $\C^n$ must be linearisable, that is, whether $\C^n$ with such an action is isomorphic to $\C^n$ with a linear $G$-action. We call the latter an $n$-dimensional $G$-module.  The first counterexamples in the algebraic setting were constructed by Schwarz \cite{Schwarz-1989b} for $n\geq 4$.  These examples are, however, holomorphically linearisable.  The first counterexamples in the holomorphic setting were given by Derksen and Kutzschebauch \cite{Derksen-Kutzschebauch}.  They showed that for every nontrivial $G$, there is an integer $N_G$ such that for all $n\geq N_G$, there is a non-linearisable effective holomorphic action of $G$ on $\C^n$.  The stratified quotients of the actions that they constructed are not isomorphic to the stratified quotient of any $G$-module.  The next main theorem states that under mild assumptions, this is the only obstruction to linearisability, that is, a holomorphic $G$-action on $\C^n$ is linearisable if its stratified quotient is isomorphic to the stratified quotient of a $G$-module.

\begin{maintheorem}   \label{mainthm:linearisation}
Let $G$ be a reductive complex Lie group.  Let $X$ be a Stein $G$-manifold and $V$ a $G$-module with common quotient $Q$.  Suppose that $V$ (or, equivalently, $X$) is large or that $X$ and $V$ are locally $G$-biholomorphic over $Q$.  Then $X$ and $V$ are $G$-biholomorphic.
\end{maintheorem}

Note that this result does more than give a solution to the linearisation problem: it provides a sufficient condition for a Stein manifold to be biholomorphic to affine space.

Let $X$ be a reduced Stein space, $G$ a complex Lie group, and $A$ a complex Lie subgroup of the automorphism group of $G$. Then given a cocycle on $X$ with values in $A$, we can produce a holomorphic group bundle $\G$ over $X$ whose fibres are (non-canonically) isomorphic to $G$. As usual,  $\G$ is said to be  \emph{trivial\/} if it is isomorphic to $X\times G$.  Let $E$ be a homogeneous holomorphic $\G$-bundle on $X$, so $\G$ acts on $E$ over $X$ such that the action of each fibre of $\G$ on the corresponding fibre of $E$ is transitive.  Ramspott proved that (when $\G$ is trivial) the inclusion of the space of holomorphic sections of $E$ over $X$ into the space of continuous sections induces a bijection of path components \cite{Ramspott}.

If $K$ is a compact real Lie group, let $K^\C$ denote its complexification, which is a reductive complex Lie group.  (Conversely, if $H$ is a reductive complex Lie group, then $H$ is isomorphic to $K^\C$ for any maximal compact real subgroup of $H$.)  Assume that $K^\C$ acts on $X$, $E$, and $\G$ compatibly with the projections to $X$ and action of $\G$ on $E$. Then we say that $\G$ is a holomorphic group $K^\C$-bundle on $X$ and that $E$ is a holomorphic $K^\C$-$\G$-bundle on $X$.  The next main theorem is an equivariant version of Ramspott's theorem, with the stronger conclusion that the inclusion is a weak homotopy equivalence. 

\begin{maintheorem}  \label{mainthm:sections}
Let $E$ be a homogeneous holomorphic $K^\C$-$\G$-bundle on a reduced Stein space $X$, where $K$ is a compact real Lie group whose complexification $K^\C$ acts on $X$, and $\G$ is a holomorphic group $K^\C$-bundle on $X$.  Then the inclusion of the space of $K^\C$-equivariant holomorphic sections of $E$ over $X$ into the space of $K$-equivariant continuous sections is a weak homotopy equivalence.
\end{maintheorem}

Note that for holomorphic sections, $K$-equivariance and $K^\C$-equivariance are equivalent.  In our context, $K$-equivariance is an appropriate condition on continuous sections; $K^\C$-equivariance is too strong.

Two special cases of Theorem \ref{mainthm:sections} are of particular interest.  First, the theorem holds when $E$ is a holomorphic principal $K^\C$-$\G$-bundle (the fibres of $\G$ act simply transitively on the fibres of $E$).  The other special case is the \lq\lq uncoupled\rq\rq\ case.  It is a parametric Oka principle for equivariant maps from a Stein $K^\C$-space to a complex homogeneous $K^\C$-space $G/H$, where the $K^\C$-action can be quite general (see the introduction to \cite{KLS-2018}).  For example, we could have $H=K^\C$ acting on $G/H$ by left multiplication.  The geometry of such an action can be quite complicated, as when $H=\SO(n,\C)$ is the subgroup of $G=\SL(n,\C)$ fixed by the holomorphic involution $A\mapsto (A^{-1})^t$ and $G/H$ is the space of symmetric bilinear forms on $\C^n$ of discriminant 1.

Theorem \ref{mainthm:sections} may be used to strengthen the main result of Heinzner and Kutzsche\-bauch \cite{Heinzner-Kutzschebauch} on the classification of principal bundles with a group action as follows.  The special case of no action is one of the central results of the Grauert era, proved by Grauert himself and improved by Cartan.

\begin{maintheorem}   \label{mainthm:principal}
Let $K$ be a compact Lie group. Suppose that $K^\C$ acts holomorphically on a reduced Stein space $X$ and on a holomorphic group bundle $\G$ on $X$.  

{\rm (a)}  Every topological principal $K$-$\G$-bundle on $X$ is topologically $K$-isomorphic to a holomorphic principal $K^\C$-$\G$-bundle on $X$.

{\rm (b)}   Let $P_1$ and $P_2$ be holomorphic principal $K^\C$-$\G$-bundles on $X$.  Every continuous $K$-isomorphism $P_1\to P_2$ can be deformed through such isomorphisms to a holomorphic $K$-isomorphism.  In fact, the inclusion of the space of holomorphic $K$-isomorphisms $P_1\to P_2$ into the space of continuous $K$-isomorphisms is a weak homotopy equivalence.
\end{maintheorem}

The sixth and final main theorem is the first and so far only equivariant version of Gromov's Oka principle.  The first part of the theorem is an equivariant version of the result that every continuous map from a Stein manifold $X$ to an Oka manifold $Y$ can be deformed to a holomorphic map.  To adapt this result to actions of a reductive group $G$ on $X$ and $Y$, we need to find the right notion of $Y$ being $G$-Oka.  We define $Y$ to be $G$-Oka if, for every reductive closed subgroup $H$ of $G$, the submanifold $Y^H$ of points fixed by $H$ is Oka in the usual sense.  (By Bochner's linearisation theorem, the subvariety $Y^H$ is indeed smooth, but of course not necessarily connected.)  Taking $H$ to be trivial, we see that a $G$-Oka manifold is Oka.  The definition is motivated in Section \ref{sec:manifolds} and suffices for an equivariant Oka principle to hold, although so far not for an arbitrary action.  It turns out that a new notion of a $G$-Stein manifold is not required -- a $G$-Stein manifold should simply be a Stein $G$-manifold -- but an Oka $G$-manifold need not be $G$-Oka (Example \ref{ex:not.G.Oka}).

\begin{maintheorem}   \label{mainthm:manifolds}
Let $G$ be a reductive complex Lie group and let $K$ be a maximal compact subgroup of $G$.  Let $X$ be a Stein $G$-manifold and $Y$ a $G$-Oka manifold.  Suppose that all the stabilisers of the $G$-action on $X$ are finite.

{\rm (a)}  Every $K$-equivariant continuous map $f:X\to Y$ is homotopic, through such maps, to a $G$-equivariant holomorphic map.  

{\rm (b)}  If $f$ is holomorphic on a $G$-invariant subvariety $Z$ of $X$, then the homotopy can be chosen to be constant on $Z$.

{\rm (c)}  If $f$ is holomorphic on a neighbourhood of a $G$-invariant subvariety $Z$ of $X$ and on a neighbourhood of a $K$-invariant $\O(X)$-convex compact subset $A$ of $X$, and $\ell\geq 0$ is an integer, then the homotopy can be chosen so that the intermediate maps agree with $f$ to order $\ell$ along $Z$ and are uniformly close to $f$ on $A$.
\end{maintheorem}

Two special cases of interest are when the group $G$ is finite and when the $G$-action on $X$ is free, so $X$ is a principal $G$-bundle.  The first examples of $G$-Oka manifolds are $G$-modules and $G$-homogeneous spaces.  A small number of other examples are known, such as any Hirzebruch surface with its natural $\GL(2, \C)$-action and, by very recent work of Kusakabe \cite{Kusakabe}, every $n$-dimensional smooth toric variety with its action of the torus $(\C^*)^n$.  The class of $G$-Oka manifolds has all the good basic properties that one would expect.  It is straightforward to make the notion of a dominating spray equivariant and thus define $G$-ellipticity, which also has all the good basic properties that one would expect and implies the $G$-Oka property.

\section{Equivariant isomorphisms}  
\label{sec:isomorphisms}

\noindent
Let $G$ be a reductive complex group and let $X$ and $Y$ be Stein $G$-manifolds.  When is there a   $G$-equivariant biholomorphism $\Phi\colon X\to Y$?  We try to reduce the question to a problem in Oka theory.  If $\Phi$ exists, then the induced map $\phi\colon X\git G\to Y\git G$ is a strata preserving biholomorphism.  Given such a map $\phi$, we can identify $X\git G$ and $Y\git G$ and we call the common quotient $Q$ with quotient morphisms denoted $\pi_X$ and $\pi_Y$.  For $U\subset Q$, let $X_U$ denote $\pi_X\inv(U)$ and similarly define $Y_U$. If $\Phi$ exists, then there is certainly an open cover of $Q$ by Stein open sets $U_i$ and $G$-equivariant biholomorphisms $\Phi_i\colon X_{U_i}\to Y_{U_i}$ which induce the identity on $U_i$.  We then say that \emph{$X$ and $Y$ are locally $G$-biholomorphic over the common quotient $Q$}.  The existence of the biholomorphism $\phi$ and Luna's slice theorem guarantee  that there are $G$-biholomorphisms $\Phi_i\colon X_{U_i}\to Y_{U_i}$, but not that each $\Phi_i$ induces the identity map of $U_i$.  \emph{For now we assume that $X$ and $Y$ are locally $G$-biholomorphic over $Q$.}  Later we will look for sufficient conditions for this to be true.

For an open subset $U$ of $Q$, let $\A(U)=\Aut_U(X_U)^G$ denote the group of $G$-biholomorphic automorphisms of $X_U$ which induce the identity on $U$.  There is an open cover $(U_i)$ of $Q$ and $G$-biholomorphisms $\Psi_i\colon X_{U_i}\to Y_{U_i}$ which induce the identity on  $U_i$.  Let $\Phi_{ij}=\Psi_i\inv\circ\Psi_j\in\A(U_i\cap U_j)$.  Then $(\Phi_{ij})$ is a cocycle, an element of $Z^1(Q,\A)$, with corresponding class $c_Y\in H^1(Q,\A)$.  If $Y'$ is also locally $G$-biholomorphic to $X$ over $Q$, then $c_Y=c_{Y'}$ if and only if there is a $G$-biholomorphism of $Y$ and $Y'$ inducing the identity on $Q$.   Conversely, given $(\Phi_{ij}')\in Z^1(Q,\A)$, there is a $G$-manifold $Y'$, locally biholomorphic to $X$ over $Q$, whose cocycle is precisely $(\Phi_{ij}')$.  By \cite[Theorem 5.11]{KLS-2017}, $Y'$ is Stein.  Thus we have the following theorem.

\begin{theorem}\label{thm:cocycles}
The isomorphism classes of Stein $G$-manifolds locally $G$-biholomorphic to $X$ over $Q$ are in bijective correspondence with $H^1(Q,\A)$.
\end{theorem}

Now suppose that $X\to Q$ and $Y\to Q$ are principal $G$-bundles, that is, the actions of $G$ on $X$ and $Y$ are free.  Then for $U$ open in $Q$, $\A(U)$ is the group of holomorphic maps of $U$ to $G$.  We also have $\CC(U)$, the group of continuous maps of $U$ to $G$.  Then $H^1(Q,\CC)$ consists, essentially, of the isomorphism classes of topological principal $G$-bundles on $Q$.  Grauert's Oka principle now has several consequences.
\begin{enumerate}
\item[(G1)]\label{g1} The natural map $H^1(Q,\A)\to H^1(Q,\CC)$ is an isomorphism.
\end{enumerate}
This implies that:
\begin{enumerate}
\addtocounter{enumi}{1}
\item[(G2)]\label{g2} If $E$ is a topological principal $G$-bundle over $Q$, then  it has a holomorphic  structure.
\item[(G3)]\label{g3} If $E$ and $E'$ are holomorphic principal $G$-bundles which are continuously isomorphic, then they are holomorphically isomorphic.
\end{enumerate}
In fact, more is true.
\begin{enumerate}
\addtocounter{enumi}{3}
\item[(G4)]\label{g4} If $\Psi\colon E\to E'$ is a continuous isomorphism of holomorphic principal $G$-bundles over $Q$, then there is a homotopy $\Psi_t$ of continuous isomorphisms of principal $G$-bundles with $\Psi_0=\Psi$ and $\Psi_1$ holomorphic.
\end{enumerate}

Returning to our more general case where $X$ and $Y$ are not necessarily principal $G$-bundles, we want to find some analogue of the sheaf of groups $\CC$ for which we can prove analogues of the results above.  Most of our work has been concentrated on proving the analogue of (G4).  
 
Our problem here is more complicated than in Grauert's case since $X\to Q$ and $Y\to Q$ are only $G$-fibre bundles over the strata of $Q$ and, moreover, the fibre of each stratum $S$ is not usually a group or even a  homogeneous space.  
 
Let $U$ be open in $Q$.  Then we have the $G$-finite functions $\O_\gf(X_U)$ on $X_U$, which are just the elements $f$ of $\O(X_U)$ such that $\{f\circ g\inv\mid g\in G\}$ spans a finite-dimensional $G$-module.  On a fibre $X_q$, the $G$-finite functions $\O_\gf(X_q)$ are a finitely generated  complex algebra. In fact, there is a finite dimensional $G$-submodule  $V\subset \O_\gf(X_q)$ which generates $\O_\gf(X_q)$.  It follows that  a $G$-biholomorphism $X\to Y$, inducing the identity on $Q$, induces algebraic $G$-isomorphisms of the fibres $X_q$ and $Y_q$, $q\in Q$. Our analogues of continuous isomorphisms of principal $G$-bundles are $G$-diffeomorphisms or $G$-homeomorphisms which behave reasonably on $G$-finite functions or on  fibres.
 
For $q\in Q$, let $(X_q)_\red$ denote the reduced structure on $X_q$.

\begin{definition}
Let $\Psi\colon X\to Y$ be a $G$-diffeomorphism inducing the identity on $Q$. We say that $\Psi$ is \emph{strict\/} if for all $q\in Q$, the induced map $\Psi\colon X_q\to Y_q$ induces an (algebraic) isomorphism of $(X_q)_\red$ and $(Y_q)_\red$.
\end{definition}

The definition is from \cite{KLS-2017}; in \cite{KLS-2015} we required isomorphisms of $X_q$ and $Y_q$ which turns out not to be necessary. Here is part (a) of Theorem \ref{mainthm:isomorphisms1}.

\begin{theorem}\label{thm:main.strict} \cite[Theorem 1.4(1)]{KLS-2017}
Let $X$ and $Y$ be Stein $G$-manifolds locally $G$-biholomorphic over  $Q$. Let $\Psi\colon X\to Y$ be a strict $G$-diffeomorphism. Then there is a continuous deformation of $\Psi$, through strict $G$-diffeomorphisms, to a $G$-biholomorphism.
\end{theorem}

Now we define \emph{strong homeomorphisms} of $X$ and $Y$ (see \cite[Section 3]{KLS-2017}  for  details). They are $G$-homeomorphisms of $X$ and $Y$, inducing the identity on $Q$, which behave well with respect to $G$-finite functions (and induce isomorphisms of the fibres $X_q$ and $Y_q$, $q\in Q$).  We start with a $G$-homeomorphism  $\Psi\colon X\to Y$ which induces the identity on $Q$.  Let $V$ be a $G$-module.  For a Stein neighbourhood $U$ of $q\in Q$, let   $\O(X_U)_V$ denote the span of the $G$-submodules of $\O_\gf(X_U)$ which are isomorphic to $V$.  For $U$ sufficiently small, $\O(X_U)_V$ is a finitely-generated $\O(U)$-module, say with generators $f_1,\dots,f_m$.  By judiciously choosing $U$ and $V$, we can assume that $\O_\gf(X_U)$ is generated by $\O(U)$ and $f_1,\dots,f_m$.  There are generators $f_1',\dots,f_m'$ of $\O(Y_U)_V$ which generate $\O_\gf(Y_U)$ as $\O(U)$-module.  We say that $\Psi$ is \emph{strong over $U$\/} if $\Psi^*f_i'=\sum a_{ij}f_j$, where the $a_{ij}$ are continuous functions on $U$ (considered as $G$-invariant functions on $X_U$).  We say that \emph{$\Psi$ is a strong $G$-homeomorphism\/} if it is strong over an open cover of $Q$.  It is not completely obvious, but the inverse of a strong $G$-homeomorphism is strong.  Here is part (b) of Theorem \ref{mainthm:isomorphisms1}.

\begin{theorem}\label{thm:main.strong} \cite[Theorem 1.4(2)]{KLS-2017}
Let $X$ and $Y$ be Stein $G$-manifolds locally $G$-biholomorphic over  $Q$. Let $\Psi\colon X\to Y$ be a strong $G$-homeomorphism.  
There is a continuous deformation of $\Psi$, through strong $G$-homeomorphisms, to a $G$-biholomorphism.
\end{theorem}
 
For $U$ open in $Q$, let $\A_s(U)$ denote the strict $G$-diffeomorphisms of $X_U$ and let $\A_c(U)$ denote the strong $G$-homeomorphisms of $X_U$, inducing the identity on $U$ in both cases.  Note that Theorems \ref{thm:main.strict} and \ref{thm:main.strong} are versions of Grauert's (G4) above, where $\CC$ is replaced by the sheaves of groups $\A_s$ and $\A_c$.  In \cite{Schwarz-2018} we actually show that (G1) holds  for $\A$ and $\A_c$. 

In \cite{KLS-2015}, we proved the analogues of (G3) for $\A_s$ and $\A_c$, under the assumption that the action of $G$ on $X$ (equivalently, on $Y$) is \emph{generic\/} (see below).  It is useful here to sketch the idea of the proof.  We say that a $G$-diffeomorphism $\Phi$ of $X$ is \emph{special\/} if it is of the form $x\mapsto\gamma(x)\cdot x$, where $\gamma\colon X\to G$ is smooth and $G$-equivariant, where the $G$-action on $G$ is by conjugation.  If $\Phi$ is holomorphic, then it is special with $\gamma$ holomorphic \cite[Lemma 6]{KLS-2015}.  Finally, if $X$ and $Y$ are locally $G$-biholomorphic over $Q$ and $(U_i)$ is an open cover of $Q$ with $G$-biholomorphisms $\Phi_i\colon X_{U_i}\to Y_{U_i}$ over $U_i$, then we say that $\Psi\colon X\to Y$ is special if $\Phi_i\inv\circ\Psi\colon X_{U_i}\to X_{U_i}$ is special for all $i$.
 
\begin{theorem}
Let $\Psi\colon X\to Y$ be a strict $G$-diffeomorphism or strong $G$-homeo\-morphism, where $X$ is generic.  Then $\Psi$ is homotopic, through $G$-isomorphisms of the same type, to a special $G$-diffeomorphism of $X$ and $Y$.
\end{theorem}

The proof of this theorem is by induction over the strata of $Q$ using local deformations of $\Psi$ to special $G$-diffeomorphisms.  Once we only need to deal with $\Psi$ special, we can use a (somewhat complicated) bundle construction to reduce proving (G3) to the equivariant Oka-Grauert principle of Heinzner-Kutzschebauch \cite{Heinzner-Kutzschebauch}.  So the plan was to deform our given isomorphism  of $X$ and $Y$ to a \lq\lq nicer\rq\rq\ one (not using an Oka principle) and then to  apply an Oka principle to deform the nice isomorphism to one that is holomorphic.

The same recipe is followed in \cite{KLS-2017}, where we define the \emph{$G$-diffeomorphisms of type $\F$\/} of $X$ and $Y$ (see below).  Let $\F$ also denote the corresponding sheaf of $G$-automorphisms of $X$ over $Q$. Then $\A\subset\F\subset\A_c$ and $\A\subset \F\subset\A_s$ and we show, as in \cite{KLS-2015}, the following.

\begin{theorem}\label{thm:deform.F}  \cite[Theorems 8.7, 8.8]{KLS-2017}
Let $\Psi\colon X\to Y$ be a strict $G$-diffeo\-morphism or strong $G$-homeomorphism. Then $\Psi$ is homotopic, through maps of the same type, to a $G$-diffeomorphism of type $\F$.
\end{theorem}

Now we are able to use the Oka-Grauert machine, as presented in \cite{Cartan} (see Section \ref{sec:sections}), to prove the following.

\begin{theorem}\label{thm:main.typeF}
Let $X$ and $Y$ be locally $G$-biholomorphic over $Q$.  
\begin{enumerate}
\item The inclusion $\A\hookrightarrow\F$ induces an isomorphism $H^1(Q,\A)\to H^1(Q,\F)$.
\item If $\Psi\colon X\to Y$ is a $G$-diffeomorphism  of type $\F$, then $\Psi$ is homotopic, through $G$-diffeomorphisms of type $\F$, to a $G$-biholomorphism.
\end{enumerate}
\end{theorem}

Before defining $G$-diffeomorphisms of type $\F$, we define the corresponding Lie algebra of vector fields of type $\LF$.  For $U$ open in $Q$, let $\Der_Q^\infty(X_U)$ (resp.\ $\Der_Q(X_U))$ denote the smooth (resp.\ holomorphic) vector fields on $X_U$ which annihilate $\O(X_U)^G$.  A vector field $D\in\Der_Q^\infty(X)$ is of type $\LF$ if every $q\in Q$ has a neighbourhood $U$ such that  $D|_{X_U}=\sum a_iA_i$ where $a_i\in\ci(X_U)^G$ and $A_i\in\Der_Q(X_U)$ for all $i$.  We let $\LF$ denote the corresponding sheaf on $Q$.  If $D\in\LF(U)$ is $\sum  a_i(x)A_i(x)$ where $a_i\in\O(U)$ and $A_i\in\Der_Q(X_U)^G$, then  $D(x,x'):=\sum a_i(x)A_i(x')$ is a family of smooth vector fields which are holomorphic for fixed $x$ and $G$-invariant in $x'$.  Then we have the following similar definition.

\begin{definition}
 Let $\Phi\colon X\to X$ be a  $G$-diffeomorphism inducing the identity on $Q$.
We say that \emph{$\Phi$ is of type $\F$\/} if for every $q\in Q$ there is a neighbourhood $U$ of $q$ and a map $\Psi\colon  X_U\times X_U \to X$ such that:
\begin{enumerate}
\item For $x\in X_U$ fixed, $\Psi(x,y)$ is a $G$-equivariant biholomorphism $\{x\}\times X_U\to X$, inducing the identity on the quotient.
\item $\Psi$ is smooth in $x$ and $y$ and $G$-invariant in $y$.
\item $\Phi(x)=\Psi(x,x)$, $x\in X_U$.
\end{enumerate}
We call $\Psi$ a \emph{local holomorphic extension of $\Phi$.}  
\end{definition}

Note that if $\Phi$ is holomorphic, then it is of type $\F$ by setting $\Psi(x,y)=\Phi(y)$. The $G$-diffeomorphisms of type $\F$ are strict and  strong.
 
If one wants to prove Theorem \ref{thm:main.typeF} using the approach of \cite{Cartan}, one needs some very basic topological properties of the sheaves $\F$ and $\LF$. See \cite[\S 6]{KLS-2017}. We list a few  of them. Let $U$ be open and Stein in $Q$.

\begin{enumerate}
\item $\LF(U)$ consists of complete vector fields and is closed in the space of $\ci$ vector fields on $X_U$, hence is a Fr\'echet space.
\item $\F(U)$ is closed in $\Diff(X_U)^G$.
\item If $D\in\LF(U)$, then $\exp(D)\in \F(U)$.
\item Let $K\subset U'\subset U$, where $K$ is compact and $U'$ has compact closure in $U$.  Then there is a neighbourhood $\Omega$ of the identity in $\F(U)$ such that any $\Phi\in\F(U)$ admits a unique logarithm $D$ in $\LF(U')$. Moreover, $\log\colon\Omega\to\LF(U')$ is continuous.
\end{enumerate}
 
We now turn to the question of when $X$ and $Y$ are locally $G$-biholomorphic over $Q$.

Let $\Der(Q)$ denote the derivations of $\O(Q)$ which preserve the strata of $Q$. That is, if $f\in\O(Q)$ vanishes on the closure of a stratum $S$ of $Q$, then so does $P(f)$ for any $P\in\Der(Q)$.  If $P\in \Der(X)^G$ and $f\in\O(Q)\simeq\O(X)^G$, then $P(f)\in\O(X)^G\simeq\O(Q)$.  It is not difficult to see that the resulting derivation $(\pi_X)_*(P)$ of $\O(Q)$ lies in $\Der(Q)$.  We say that $\pi_X$ (or just $X$) has the \emph{infinitesimal lifting property\/} (abbreviated ILP), if $(\pi_X)_*$ maps onto $\Der(Q)$.

The ILP is a consequence of various more geometric conditions. Assuming that $Q$ is connected, there is a unique open and dense stratum $Q_\pr$, the principal stratum. Let $X_\pr$ denote $\pi_X\inv(Q_\pr)$.  We say that $X$ is \emph{$k$-principal\/} if $X\setminus X_\pr$ has codimension at least $k$ in $X$.  We say that the $G$-action is  \emph{stable\/} if $X_\pr$ consists of closed orbits.  If $X$ is stable and $k$-principal, $k\geq 2$, then one can reduce all our $G$-isomorphism problems to the case that $X_\pr$ consists of (closed) $G$-orbits with trivial stabilizer \cite[Proposition 3]{KLS-2015}, in which case we say that  \emph{$X$ has TPIG} (\lq\lq trivial principal isotropy groups\rq\rq).  Finally, if $X$ is $3$-principal with FPIG (\lq\lq finite principal isotropy groups\rq\rq), then $\pi_X$ has the ILP \cite[Theorem 8.9]{Schwarz-1995}.  We will see another condition implying the ILP below. Finally, $X$ is \emph{generic\/} if $X$ is $2$-principal with TPIG.

In \cite[\S 5]{KLS-2017} some local analysis establishes Theorem \ref{mainthm:isomorphisms2}. 

\begin{theorem}
Let $X$ and $Y$ have common quotient $Q$. Suppose that $\Psi\colon X\to Y$ is a strict diffeomorphism or $\Psi$ is a strong $G$-homeomorphism and $X$ has the ILP. Then $X$ and $Y$ are locally $G$-biholomorphic over $Q$.  Hence there is a homotopy $\Psi_t$ of strict or strong $G$-isomorphisms $\Psi_t$ with $\Psi_1$ biholomorphic.
\end{theorem}

In our earlier work \cite{KLS-2015}, we assume that $X$ and $Y$ are locally $G$-biholomorphic over the common quotient $Q$ and generic.  In that case we prove the following somewhat provocative result (which does not require $X$ and $Y$ to be smooth).  It relies on the equivariant Oka-principle of \cite{Heinzner-Kutzschebauch}.

\begin{theorem}
Let $G$ act holomorphically and generically on normal Stein spaces $X$ and $Y$ which are locally $G$-biholomorphic over a common quotient $Q$. Then the obstruction to $X$ and $Y$ being $G$-equivariantly biholomorphic is purely topological. Namely, there is a bundle arising from the data whose topological triviality is equivalent to $X$ and $Y$ being $G$-biholomorphic.
\end{theorem}

\section{The linearisation problem}  
\label{sec:linearisation}

\noindent
Let the complex reductive group $G$ act holomorphically on $X=\C^n$.  We say that the $G$-action is \emph{holomorphically linearisable\/} if there is a $G$-biholomorphism $\Phi\colon X\to V$, where $V$ is an $n$-dimensional 
$G$-module.  Then $\Phi$ induces a strata preserving biholomorphism $\phi\colon X\git G\to V\git G$.  Thus linearisability implies that $X\git G$ is isomorphic to the quotient of a linear $G$-action. So we have the following special case of the problem we have considered above.

\emph{Let $X$ be a Stein $G$-manifold and $V$ a $G$-module with common quotient $Q$. When are $X$ and $V$ equivariantly biholomorphic?}

Let $X_{(n)}=\{x\in X: \dim G_x=n\}$.  We say that \emph{$X$ is large\/} if $X$ is generic and $\codim_X X_{(n)}\geq n+2$ for all $n\geq 1$. Largeness holds for ``most''  $X$ and $V$. See \cite[Remark 2.1]{KLS-2017-lin}. If $X\git G$ and $V\git G$ are strata preserving biholomorphic, then $X$ is large if and only if $V$ is large.

Below we need to distinguish between $X\git G$ and $Q=V\git G$ even though they are assumed to be stratified biholomorphic.
The following implies Theorem \ref{mainthm:linearisation}.

\begin{theorem}\label{thm:main.lin}
Let $X$ be a Stein $G$-manifold and $V$  a $G$-module with a strata preserving biholomorphism $\phi\colon X\git G\to Q=V\git G$. 
\begin{enumerate}
\item If $X$ and $V$ are locally $G$-biholomorphic over $Q$, then they are $G$-biholomorphic \cite[Theorem 1.1]{KLS-2017-lin}.
\item If $V$ is large, then by perhaps changing $\phi$ by an element of $\Aut(Q)$, one can arrange that $X$ and $V$ are locally $G$-biholomorphic over $Q$, hence $G$-biholomorphic \cite[Theorem 1.4]{KLS-2017-lin}.
\end{enumerate}
\end{theorem}
 
The largeness of $V$ in (2)  is only important because it implies other properties of $V$.

Let $p_1,\dots,p_d$ be homogeneous generators of  $\C[V]^G$ of degrees $e_1,\dots,e_d$.  Let $t\in \C^*$ act on $\C^d$ by $(y_1,\dots,y_d)\mapsto (t^{e_1}y_1,\dots,t^{e_d}y_d)$.  The $\C^*$-action preserves $Q\simeq p(V)$, where $p=(p_1,\dots,p_d)\colon V\to \C^d$.  We say that $V$ has the \emph{deformation property (DP)\/} if for any $\theta\in\Aut(Q)$ fixing $0$, the limit $\theta_t=t\inv\circ\theta\circ t$  exists as $t\to 0$.  The limit $\theta_0$ is in $\Aut_\ql(Q)$, the set of \emph{quasilinear\/} elements of $\Aut(Q)$, that is, those that commute  with the $\C^*$-action.  We say that \emph{$V$ has the lifting property (LP)\/} if any $\theta\in\Aut(Q)$ has a holomorphic lift $\Theta\colon V\to V$.  The lift need not be $G$-equivariant, but it has to map $V_q$ to $V_{\theta(q)}$ for all $q\in Q$.

\begin{proposition} Suppose that $V$ is large. Then:
\begin{enumerate}
\item $V$ has the ILP \cite[Theorem 0.4]{Schwarz-1995}. In fact, any holomorphic differential operator on $Q$ lifts.
\item $V$ has the DP \cite[Theorem 2.2]{Schwarz-2014}. (The condition that $V$ be admissible in the cited theorem is implied by $V$ being large.)
\end{enumerate}
For any $V$, LP implies DP.
\end{proposition}

Here are some results of \cite{KLS-2017-lin}, which point out the inner workings of the proof of Theorem \ref{thm:main.lin}. 

\begin{theorem}\label{thm:main.lin2}
Let $X$ be a Stein $G$-manifold and $V$  a $G$-module with a strata preserving biholomorphism $\phi\colon X\git G\to Q=V\git G$. 
\begin{enumerate}
\item Let $E$ be the Euler vector field on $V$.  Then $(\pi_V)_*(E)\in\Der(Q)$ can be considered as an element of $\Der(X\git G)$ via $\phi\inv$. Suppose that it has a lift $B\in\Der(X)^G$ and that
 we have a $G$-biholomorphic lift $\Phi$ of $\phi$ over a neighbourhood of $\phi\inv(0)\in X\git G$.  Then $\phi$ lifts to a $G$-biholomorphism of $X$ and $V$ \cite[Remark 3.6]{KLS-2017-lin}.
\item If $V$ has the ILP and DP, then by perhaps changing $\phi$ by an automorphism of $Q$, $\phi$ lifts to a $G$-biholomorphism of $X$ and $V$ \cite[Section 5]{KLS-2017-lin}.
\end{enumerate}
\end{theorem}

The largeness condition of Theorem  \ref{thm:main.lin}(2) applies to most $G$-modules. The remaining $G$-modules are \lq\lq small\rq\rq.  For small modules, we have applied the criteria  of Theorem \ref{thm:main.lin2} in the following cases.

\begin{theorem}  \label{thm:small-modules}
Let $X$ be a Stein $G$-manifold and $V$ a $G$-module with stratified biholomorphism $\phi\colon X\git G\to Q=V\git G$.  In each of the following cases, by perhaps changing $\phi$ by an automorphism of $Q$, $\phi$ lifts to a $G$-biholomorphism of $X$ and $V$.
\begin{enumerate}
\item $\dim Q\leq 1$.
\item $G=\SL_2$ or $\SO_3$.
\item $G$ is finite.
\item $G^0=\C^*$.
\end{enumerate}
\end{theorem}

Parts (1) and (2) are in \cite{KLS-2017-lin} and it is rather easy to show that the relevant $G$-modules $V$ are large or have the ILP and DP.  Part (3) is easy and is \cite[Theorem 3.8]{Kutzschebauch-Schwarz} (and should have been noted well before!).  Part (4)
is much more difficult and is established in \cite{Kutzschebauch-Schwarz}.  So let us suppose that $G^0=\C^*$.  There are three \lq\lq easy cases\rq\rq\ where all the nonzero weights of $\C^*$ acting on $V$ have the same sign, there are at least two weights of each sign, or $\dim V=2$.  We are then able to reduce to the techniques and theorems above. The hard part is if there is, say, only one strictly positive weight and  at least two strictly negative weights. 

Note that $X^G$ and $V^G$ are strata of $X\git G$ and $Q$.  It is not hard to reduce to the case that $\phi\colon X^G\to V^G$ is the identity.  Then one establishes the following proposition.

\begin{proposition}
Let $\theta\in\Aut(U)$ where $U$ is a connected neighbourhood of $V^G\subset Q$ (resp.\ $V^{G^0}/G\subset Q$) and $\theta$ is the identity on $V^G$ (resp.\ $V^{G^0}/G$).  Then, modulo $\Aut(Q)$, $\theta$ has a $G$-equivariant lift $\Theta$ to $\pi_V\inv(U)$.
\end{proposition}

Using the proposition we are able to lift $\phi$ (after changing by some elements of $\Aut(Q)$) to a $G$-biholomorphism $\Phi$ over a neighbourhood $U_0$ of $X^{G^0}/G\subset X\git G$.  Let $B$ denote  $(\pi_V)_*(E)$ considered as an element of $\Der(X\git G)$ via $\phi\inv$.  Via $\Phi\inv$ applied to $E$ we have a holomorphic $G$-invariant vector field on $X_{U_0}$ which lifts $B$.  Away from $X_{U_0}$, the isotropy groups of closed orbits in $X$ are finite and we can find local $G$-invariant holomorphic lifts of $B$. Since $X\git G$ is Stein, there is a $G$-invariant holomorphic lift of $B$ to $X$.  Now apply Theorem \ref{thm:main.lin2}(1).

\section{Equivariant sections of bundles of homogeneous spaces}  
\label{sec:sections}

\noindent
In \cite{KLS-2018}, we combined many of the results on the Oka principle from the Grauert era into a single theorem in the homotopy-theoretic language of modern Oka theory.  Moreover, we generalised these results to an equivariant setting, with respect to a holomorphic action of a reductive complex Lie group.  Recall that complexification defines an equivalence of categories from compact real Lie groups to reductive complex Lie groups.  Throughout this section, $K$ denotes a compact real Lie group with complexification $K^\C$.

A special case of this equivariant setting had been considered before by Heinzner and Kutzschebauch \cite{Heinzner-Kutzschebauch}, motivated by the negative solution of the algebraic linearisation problem by Schwarz \cite{Schwarz-1989b}.  He constructed algebraic $K^\C$-vector bundles of representation spaces which are not $K^\C$-trivial and thus obtained non-linearisable algebraic actions on their total spaces.  These total spaces are isomorphic to affine spaces.  The relevant corollary of Heinzner and Kutzschebauch's work is that, unlike the algebraic situation, holomorphic $K^\C$-bundles over representation spaces are always $K^\C$-trivial, so the action on the total space is holomorphically linearisable.

Our setting is as follows.  Let $X$ be a reduced Stein space.  Let $G$ be a complex Lie group and $\G$ be a holomorphic group bundle on $X$ with fibre $G$.  By definition, $\G$ is defined by a holomorphic cocycle with respect to some open cover of $X$ with values in a complex Lie subgroup $A$ of the Lie automorphism group of $G$.  We call $A$ the structure group of $\G$ and we call $\G$ a holomorphic group $A$-bundle.  Let $\H$ be a holomorphic group subbundle of $\G$, whose fibre is a closed subgroup $H$ of $G$, so $\G$ may in fact be defined by a holomorphic cocycle with values in the group of Lie automorphisms of $G$ that preserve $H$.  Thus we assume that $A$ preserves $H$.

Let $P$ be a holomorphic principal bundle on $X$ with structure group bundle $\G$ acting from the right -- we call $P$ a principal $\G$-bundle -- and let $E$ be the quotient bundle $P/\H$.  Then $E$ is a holomorphic fibre bundle on $X$ with fibre $G/H$ (left cosets) and structure group bundle $\G$ acting from the left.  Each fibre of $\G$ acts transitively on the corresponding fibre of $E$.  We call $E$ a homogeneous $\G$-bundle.  The principal bundle $P$ is defined by a holomorphic $\G$-valued cocycle, which tells us how to form $P$ by glueing together pieces of $\G$ over an open cover of $X$.  The same cocycle encodes how $E$ may be constructed from the quotient bundle $\G/\H$ (left cosets).  Note that the action of $\G$ on $P$ need not descend to an action on $E$ (right multiplication does not respect left cosets).

Now we describe the $K^\C$-actions in our setting.  Let $K^\C$ act holomorphically on $X$, and holomorphically and compatibly on $\G$ by group $A$-bundle maps (which preserve $\H$).  This means that $K^\C$ acts on the fibres of $\G$ by elements of $A$, which makes sense because each fibre of $\G$ is canonically identified with $G$ modulo $A$.  Let $K^\C$ also act holomorphically and compatibly on $P$ such that the action map $P\times_X \G\to P$ is $K^\C$-equivariant.  We call $P$ with such an action a principal $K^\C$-$\G$-bundle.  The action of $K^\C$ on $P$ descends to an action on $E$.  We summarise all the above data by referring to $E$ as a homogeneous $K^\C$-$\G$-bundle.  

Viewed as a holomorphic fibre bundle with fibre $G$, the bundle $P$ can be taken to have the semidirect product $A\ltimes G$ as its structure group.  Equivariance of the action map $P\times_X \G\to P$ is equivalent to $K^\C$ acting on $P$ by $A\ltimes G$-bundle maps, meaning that $K^\C$ acts on the fibres of $P$ by elements of $A\ltimes G$.  If $P'$ is another holomorphic principal $K^\C$-$\G$-bundle, then the holomorphic group bundle $\Aut\, P$ with fibre $G$ and the holomorphic principal bundle $\Iso(P',P)$ with fibre $G$ and structure group bundle $\Aut\, P$ have induced structure groups that are complex Lie groups and they have induced $K^\C$-actions by elements of the respective structure group that make the action map $\Iso(P',P)\times_X \Aut\, P\to\Iso(P',P)$ equivariant.  All spaces of sections are endowed with the compact-open topology.

\smallskip

The main result of \cite{KLS-2018} is Theorem \ref{mainthm:sections} from the introduction.  As described above, we have a homogeneous holomorphic $K^\C$-$\G$-bundle $E$ on the reduced Stein space $X$, where $K$ is a compact real Lie group and $\G$ is a holomorphic group $K^\C$-bundle on $X$.  The theorem states that the inclusion of the space of $K^\C$-equivariant holomorphic sections of $E$ over $X$ into the space of $K$-equivariant continuous sections is a weak homotopy equivalence.

The proof of Theorem \ref{mainthm:sections} follows the approach of the Grauert era, clearly and elegantly described by Cartan in \cite{Cartan}.  This approach has the advantage that we can apply results from Heinzner and Kutzschebauch's paper \cite{Heinzner-Kutzschebauch}.  Let us review some notions central to this approach, adapted to the equivariant setting.  First we need the Kempf-Ness set. 

To every real-analytic $K$-invariant strictly plurisubharmonic exhaustion function on $X$ (such functions exist) is associated a real-analytic subvariety $R$ of $X$ called a Kempf-Ness set.  It consists of precisely one $K$-orbit in every closed $K^\C$-orbit in $X$.  The inclusion $R\hookrightarrow X$ induces a homeomorphism $R/K\to X\git K^\C$, where the orbit space $R/K$ carries the quotient topology.  Informally speaking, the Kempf-Ness set is where $K$-equivariant continuous information can be related to holomorphic $K^\C$-equivariant information.  This is underlined by the following result, which in its original form is due to Neeman \cite{Neeman-1985}; see also \cite{Schwarz-1989b} and \cite{Heinzner-Huckleberry}.

\begin{theorem}  \cite[p.~341]{Heinzner-Kutzschebauch}   \label{thm:HK-retraction}
There is a real-analytic $K$-invariant strictly plurisubharmonic exhaustion function on $X$, whose Kempf-Ness set $R$ is a $K$-equivariant continuous strong deformation retract of $X$, such that the deformation preserves the closure of each $K^\C$-orbit.
\end{theorem}

The following equivariant version of the covering homotopy theorem is used to prove an important fact about the Kempf-Ness set (Proposition \ref{prp:topological.fact} below). 

\begin{theorem}  \cite[Theorem 2.6]{KLS-2018}
\label{thm:homotopy.theorem}
Let a compact Lie group $K$ act real-analytically on a Stein space $X$ by biholomorphisms, and trivially on $I=[0,1]$.  Let $G$ be a complex Lie group and $\G$ be a topological group bundle on $X\times I$ with fibre $G$, whose structure group $A$ is a Lie subgroup of the Lie automorphism group of $G$.  Let $K$ act continuously on $\G$ by group $A$-bundle maps.

{\rm (a)}  Then $\G$ is isomorphic to a constant bundle (depends trivially on $t\in I$).

{\rm (b)}  Let $P$ be a topological principal $K$-$\G$-bundle on $X\times I$.  (It is implicit that the action map $P\times_X \G\to P$ is $K$-equivariant.)  By {\rm (a)}, we may take $\G$ to be constant.  Then $P$ is isomorphic to a constant bundle.  Hence, once we identify the bundles $\G\vert_{X\times\{t\}}$, $t\in I$, with a topological group $K$-bundle $\G_0$ on $X$, the topological principal $K$-$\G_0$-bundles $P\vert_{X\times\{t\}}$, $t\in I$, are mutually isomorphic.
\end{theorem}

As a consequence we have the following result.

\begin{proposition}   \cite[Proposition 2.7]{KLS-2018}
\label{prp:topological.fact}
Suppose that a compact Lie group $K$ acts real-analytically on a Stein space $X$ by biholomorphisms.  Let $G$ be a complex Lie group and $\G$ be a topological group bundle on $X$ with fibre $G$, whose structure group $A$ is a Lie subgroup of the Lie automorphism group of $G$.  Let $K$ act continuously on $\G$ by group $A$-bundle maps.

Let $E$ be a topological $K$-$\G$-bundle on $X$ (not necessarily homogeneous).  The restriction map from the space of continuous $K$-sections of $E$ over $X$ to the space of continuous $K$-sections of $E$ over $R$ is a homotopy equivalence.
\end{proposition}

In the following, we take $R$ to be a Kempf-Ness set as defined above.  Next comes the central notion of an $NHC$-section.

Let $C$ be a compact Hausdorff space and $N\subset H$ be closed subsets of $C$, such that $N$ is a strong deformation retract of $C$.  We define a sheaf $\Q(R)$ of topological groups on $X\git K^\C$ as follows.  For each open subset $V$ of $X\git K^\C$, the group $\Q(R)(V)$ consists of all $K$-equivariant $NHC$-sections of $\G$ over $W=(\pi^{-1}(V)\times H)\cup((\pi^{-1}(V)\cap R)\times C)$.  By an $NHC$-section of $\G$ over $W$, we mean a continuous map $s:W\to\G$ such that:
\begin{itemize}
\item  for every $c\in C$, the map $s(\cdot,c)$ is a continuous section of $\G$ over $\pi^{-1}(V)\cap R$,
\item  for every $c\in H$, $s(\cdot,c)$ is a holomorphic section of $\G$ over $\pi^{-1}(V)$,
\item  for every $c\in N$, $s(\cdot,c)$ is the identity section of $\G$ over $\pi^{-1}(V)$.
\end{itemize}
The topology on $\Q(R)(V)$ is the compact-open topology.

Now we formulate the relevant results from \cite{Heinzner-Kutzschebauch}, first the equivariant analogue of the classical {\it th\'eor\`eme principal} \cite[p.~105]{Cartan}.

\begin{theorem}  \cite[p.~324]{Heinzner-Kutzschebauch}   \label{thm:HK-NHC}
\begin{enumerate}
\item[(a)]  The topological group $\Q(R)(X\git K^\C)$ is path connected.
\item[(b)]  If $U$ is Runge in $X\git K^\C$, then the image of $\Q(R)(X\git K^\C)$ in $\Q(R)(U)$ is dense.
\item[(c)]  $H^1(X\git K^\C,\Q(R))=0$.
\end{enumerate}
\end{theorem}

Next we state Heinzner and Kutzschebauch's main result on the classification of principal $K$-$\G$-bundles (called $\G$-principal $K$-bundles in \cite{Heinzner-Kutzschebauch}).

\begin{theorem}  \cite[p.~341, 345]{Heinzner-Kutzschebauch}   \label{thm:HK-classification}
{\rm (a)}  Every topological principal $K$-$\G$-bundle on $X$ is topologically $K$-isomorphic to a holomorphic principal $K^\C$-$\G$-bundle on $X$.

{\rm (b)}  Let $P_1$ and $P_2$ be holomorphic principal $K^\C$-$\G$-bundles on $X$.  Let $c$ be a continuous $K$-equivariant section of $\Iso(P_1,P_2)$ over $R$.  Then there is a homotopy of continuous $K$-equivariant sections $\gamma(t)$, $t\in[0,1]$, of $\Iso(P_1,P_2)$ over $R$, such that $\gamma(0)=c$ and $\gamma(1)$ extends to a holomorphic $K$-equivariant isomorphism from $P_1$ to $P_2$.
\end{theorem}

\begin{proof}[Sketch of proof of Theorem \ref{mainthm:sections}]
First we prove that the inclusion $\Gamma_\O(E)^K\hookrightarrow\Gamma_\mathscr C(E)^K$ induces a surjection of path components.  Let $P$ be the holomorphic principal $K^\C$-$\G$-bundle associated to $E$.  Take a continuous $K$-section $s$ of $E$ over $X$.  The preimage in $P$ of its image in $E$ is a topological principal $K$-$\H$-subbundle $Q$ of $P$.  We have a topological $K$-isomorphism $\sigma:Q\times^\H\G \to P$.  By Theorem \ref{thm:HK-classification}(a), $Q$ is topologically $K$-isomorphic to a holomorphic principal $K^\C$-$\H$-bundle $Q'$.  Choose a topological $K$-isomorphism $Q'\to Q$ and let $\tau:Q'\times^\H\G \to Q\times^\H\G$ be the induced isomorphism.  

By Theorem \ref{thm:HK-classification}(b) and Proposition \ref{prp:topological.fact}, the topological $K$-isomorphism $\sigma\circ\tau:Q'\times^\H\G \to P$ can be deformed to a holomorphic $K$-isomorphism over $X$.  Applying the deformation to $Q'$, viewed as a subbundle of $Q'\times^\H\G$, gives a deformation of $Q$ through topological principal $K$-$\H$-subbundles of $P$ to a holomorphic principal $K^\C$-$\H$-subbundle.  Pushing down to $E$ yields a deformation of $s$ through continuous $K$-sections of $E$ to a holomorphic section.

Now let $B$ be the closed unit ball in $\R^k$, $k\geq 1$, and let $\alpha_0:B\to \Gamma_\mathscr C(E)^K$ be a continuous map taking the boundary sphere $\partial B$ into $\Gamma_\O(E)^K$.  Choose a base point $b_0\in \partial B$.  We shall prove that there is a deformation $\alpha:B\times I\to\Gamma_\mathscr C(E)^K$ of $\alpha_0=\alpha(\cdot,0)$ with $\alpha_t(b_0)=\alpha_0(b_0)$ and $\alpha_t(\partial B)\subset\Gamma_\O(E)^K$ for all $t\in I$, and $\alpha_1(B)\subset\Gamma_\O(E)^K$.  This implies that the inclusion $\Gamma_\O(E)^K\hookrightarrow\Gamma_\mathscr C(E)^K$ induces a $\pi_{k-1}$-monomorphism and a $\pi_k$-epimorphism.

Consider the holomorphic group $K^\C$-bundle $\Aut\, P$ of principal $\G$-bundle automorphisms of $P$.  We seek a global $K$-equivariant $NHC$-section $\gamma_0$ of $\Aut\, P$ (with $C=B$, $H=\partial B$, $N=\{b_0\}$) such that for every $b\in B$, $\gamma_0(b)$, by its left action on $E$, maps $\alpha_0(b_0)$ to $\alpha_0(b)$, over $X$ if $b\in\partial B$ but only over $R$ if $b\in B\setminus \partial B$.

\noindent
\textbf{Claim.}  On a sufficiently small saturated neighbourhood of each point of $X$, that is, locally over $X\git K^\C$, such an $NHC$-section exists.

The proof of the claim is quite involved.  It requires a detailed analysis of the equivariant local structure of the bundles involved.  We will not attempt a summary, but refer the reader to \cite{KLS-2018}.

On the intersection of two such saturated neighbourhoods, two such $NHC$-sections differ by a $K$-equivariant $NHC$-section of the holomorphic group $K^\C$-bundle $\A$ of principal $\G$-bundle automorphisms of $P$ that fix $\alpha_0(b_0)$.  Gluing these local $NHC$-sections together to produce $\gamma_0$ amounts to splitting a cocycle, and the cocycle does split by Theorem \ref{thm:HK-NHC}(c) applied to $\A$.

By Theorem \ref{thm:HK-NHC}(a) applied to $\Aut\, P$, we can deform $\gamma_0$ through $K$-equivariant $NHC$-sections $\gamma_t$ of $\Aut\, P$, $t\in I$, to the identity section.  Let $\alpha_t(b)$ be the section of $E$ obtained by letting $\gamma_t(b)$ act on $\alpha_0(b_0)$.  For $b\in B\setminus \partial B$ and $t\in(0,1)$, $\alpha_t(b)$ is only defined over $R$.  Thus we have a deformation $\alpha:B\times I\to\Gamma_\mathscr C(E\vert_R)^K$, such that $\alpha$ factors through $\Gamma_\mathscr C(E)^K$ on $B\times\{0\}$ and through $\Gamma_\O(E)^K$ on $\partial B\times I\cup B\times\{1\}$, in such a way that $\alpha_t(b_0)$ is fixed and $\alpha_1$ takes all of $B$ to $\alpha_0(b_0)$.

At this stage the continuous sections are defined over the Kempf-Ness set only, whereas the holomorphic sections are defined over all of $X$.  The problem is that  the extension of continuous sections  using the strong deformation retraction of $X$ onto $R$  does not give back the  holomorphic sections for parameters in $\partial B$.

The proof would now be done if we could show that the following  commuting square has a continuous lifting.
\[ \xymatrix{
\partial B\times I\cup B\times\{0,1\} \ar^{\phantom{mmmmm}\beta}[r] \ar^j[d] & \Gamma_\mathscr C(E)^K \ar^p[d]  \\ B\times I \ar^\alpha[r] \ar@{-->}[ur] & \Gamma_\mathscr C(E\vert_R)^K
} \]
In fact, it suffices to show that $\alpha$ can be deformed, keeping $\beta$ fixed and the square commuting, until a lifting exists.  This can be deduced by homotopy-theoretic considerations from Proposition \ref{prp:topological.fact}.
\end{proof}

\section{Equivariantly Oka manifolds}  
\label{sec:manifolds}

\noindent
We begin by motivating the definition of a $G$-Oka manifold.  Here $G$ is a reductive complex Lie group.  Let $K$ be a maximal compact subgroup of $G$.  It is natural to say that a $G$-manifold $Y$ has the basic $G$-Oka property ($G$-BOP) if every continuous $K$-map from a Stein $G$-manifold $X$ to $Y$ can be deformed through such maps to a holomorphic map.

Consider the following consequences of $Y$ satisfying $G$-BOP.  First, if the $G$-action on $X$ is trivial, then a $K$-map $X\to Y$ is nothing but a plain map from $X$ to the submanifold $Y^K=Y^G$.  Hence, every continuous map $X\to Y^G$ can be deformed to a holomorphic map, so $Y^G$ has the basic Oka property (BOP).

Second, let $L$ be a closed subgroup of $K$.  The complexification of $L$ is a reductive closed subgroup $H$ of $G$.  Let $X$ be a Stein $H$-manifold and consider the adjunction
\[ \hom_G(\ind_H^G X, Y) \cong \hom_H(X, \res_H^G Y).\]
Here, the subscripts denote equivariance, $\hom$ refers to either continuous or holomorphic maps, $\res_H^G Y$ is $Y$ viewed as an $H$-manifold, $\ind_H^G X$ is the Stein $G$-manifold $G\times^H X$ (the geometric quotient of $G\times X$ by the $H$-action $h\cdot(g,x)=(gh^{-1}, hx)$), and $\cong$ denotes a homeomorphism that is natural in $X$ and $Y$.  We conclude that if $Y$ satisfies $G$-BOP, then $Y$ also satisfies $H$-BOP, so by the above, $Y^H$ satisfies BOP.

Approximation and interpolation can easily be included in the above and we are led to the following definition.

\begin{definition}
Let a reductive complex Lie group $G$ act holomorphically on a manifold $Y$.
 We say that $Y$ is $G$-\emph{Oka} if the fixed-point manifold $Y^H$ is Oka for all reductive closed subgroups $H$ of $G$.
\end{definition}

Taking $H$ to be the trivial subgroup, we see that a $G$-Oka manifold is Oka.  On the other hand, the following example shows that an Oka $G$-manifold need not be $G$-Oka, even for $G=\Z_2$.

\begin{example}  \cite[Example 2.7]{KLS-2021}\label{ex:not.G.Oka}
If $f\in\O(\C^n)$, $n\geq 2$, is a polynomial such that $df$ vanishes nowhere on $f^{-1}(0)$, then the affine algebraic manifold $X=\{(u,v,z)\in \C^{n+2} : uv=f(z)\}$ has the algebraic density property and is therefore Oka \cite{KK-2008}.  The fixed point set $W$ of the involution $u \leftrightarrow v$ of $X$ is smooth, given by the formula $u^2=f(z)$, and is a double branched covering of $\C^n$ with branch locus $f^{-1}(0)$.  Choose $f$ such that $f^{-1}(0)$ is not Oka; this is easy.  If $W$ is Oka, then our promised example is $Y=W$ and $Y^G=f^{-1}(0)$.  If it is not, then the example is $Y=X$ and $Y^G=W$.  (There is no particular $f$ for which we have determined whether $W$ is Oka or not.)
\end{example}

The expected basic properties of the equivariant Oka property are easily established straight from the definition or from basic properties of Oka manifolds.

\begin{proposition} \cite[Proposition 2.1]{KLS-2021}
Let a complex reductive group $G$ act holomorphically on a complex manifold $Y$.
\begin{enumerate}
\item  If $G$ acts trivially on $Y$, then $Y$ is $G$-Oka if and only if $Y$ is Oka.
\item  If $Y$ is $G$-Oka and $H$ is a reductive closed subgroup of $G$, then $Y$ is $H$-Oka with respect to the restriction of the action to $H$.
\item  If $Y_j$ is $G_j$-Oka, $j=1,2$, then $Y_1\times Y_2$ is $G_1\times G_2$-Oka.  
\item  If $Y_1$ and $Y_2$ are $G$-Oka, then $Y_1\times Y_2$ is $G$-Oka with respect to the diagonal action.
\item  A holomorphic $G$-retract of a $G$-Oka manifold is $G$-Oka.
\item  If $Y$ is the increasing union of $G$-Oka $G$-invariant domains, then $Y$ is $G$-Oka.
\end{enumerate}
\end{proposition}

Here are three ways to construct new equivariantly Oka manifolds from old.  The first two are from \cite{KLS-2021}.  The localisation principle in (c) is due to Kusakabe \cite[Theorem A.5]{Kusakabe}.

\begin{proposition}
Let $G$ be a complex reductive group and let $H$ be a reductive closed subgroup of $G$.

{\rm (a)}  If $Y$ is  an $H$-manifold, then $G\times^H Y$ with its natural $G$-action is $G$-Oka if and only if $Y$ is $H$-Oka.

{\rm (b)}  Let $\pi:Y\to Z$ be a holomorphic fibre bundle with fibre $F$.  Assume that $Y$, $Z$, and $F$ are $G$-manifolds and $\pi$ is $G$-equivariant.  Further assume that $Z$ is Stein and $F$ is $G$-Oka.  Then $Y$ is $G$-Oka if and only if $Z$ is $G$-Oka.

{\rm (c)}  If $Y$ is covered by $G$-Oka $G$-invariant Zariski-open subsets, then $Y$ is $G$-Oka.  (Zariski-open means that the complement is a closed analytic subvariety.)
\end{proposition}

The first main theorem of \cite{KLS-2021} is Theorem \ref{mainthm:manifolds} for the case when $G=K$ is a finite group.  Let us give a rough sketch of the proof of part (a).  Parts (b) and (c) then follow by equivariant adaptations of standard methods.  The proof of Theorem \ref{mainthm:manifolds} is completed in \cite[Section 5]{KLS-2021}.

\begin{proof}[Sketch of proof of part (a) of Theorem \ref{mainthm:manifolds}]  Since $G$ is finite, the Luna strata in $Q=X/G$ are finite in number.  Let $\pi:X\to Q$ be the quotient map.  We have a filtration $Q=Q_m\supset Q_{m-1}\supset\cdots\supset Q_0\supset Q_{-1}=\varnothing$ of $Q$, where the subvariety $Q_k$ is the union of the strata of dimension at most $k$.  Each difference $Q_k\setminus Q_{k-1}$, $k=0,\ldots,m$, is smooth and each of its connected components is contained in a Luna stratum.  We will produce a homotopy of continuous $G$-maps from a given map $f:X\to Y$ to a holomorphic map.  We let $f_0=f$ on $\pi^{-1}(Q_0)$ and proceed by induction in two steps for each $k=1,\ldots,m$.

\noindent\textit{Step 1.}  Suppose that we have a homotopy of $f\vert_{\pi^{-1}(Q_{k-1})}$, through continuous $G$-maps, to a holomorphic $G$-map $f_{k-1}:\pi^{-1}(Q_{k-1}) \to Y$.  Then $f_{k-1}$ and the homotopy extend to a $G$-invariant neighbourhood of $\pi^{-1}(Q_{k-1})$ in $X$.

For $k=1$, we start with the constant homotopy.  For $k\geq 2$, the input is Step 2 for $k-1$.  Using Siu's Stein neighbourhood theorem and Heinzner's equivariant embedding theorem, we move the problem into a $G$-module.  The holomorphic map extends by Cartan's extension theorem followed by averaging.  Finally, the homotopy extends since $Y$ is an absolute neighbourhood retract in the category of metrisable $G$-spaces.

\noindent\textit{Step 2.}  There is a homotopy of $f\vert_{\pi\inv(Q_k)}$, through continuous $G$-maps, to a holomorphic $G$-map $f_k:\pi^{-1}(Q_k)\to Y$.

This step may be reduced to an application of Forstneri\v c's Oka principle for sections of branched holomorphic maps (\cite[Theorem 2.1]{Forstneric-2003}; see also \cite[Theorem 6.14.6]{Forstneric-book}).  Forstneri\v c's result is the only known Oka principle in modern Oka theory that does not require the map in question to be a submersion.  A parametric version of the result is not available and appears difficult to prove.
\end{proof}

The remarkable fact that Oka manifolds can be defined in many nontrivially equi\-valent ways points to the concept being natural and important.  The same has been proved to some extent in the equivariant setting.  Above we defined the basic $G$-Oka property.  It is the property ascribed to the $G$-Oka manifold $Y$ in part (a) of Theorem \ref{mainthm:manifolds}.  The stronger property ascribed to $Y$ in part (b) is called the basic $G$-Oka property with interpolation ($G$-BOPI), and the property ascribed to $Y$ in part (c) is called the basic $G$-Oka property with approximation and jet interpolation ($G$-BOPAJI).  The definition of the basic $G$-Oka property with approximation ($G$-BOPA) should be obvious.  The following result combines \cite[Corollary 4.2]{KLS-2021} and \cite[Corollary A.4]{Kusakabe}.  The property $G$-$\textrm{Ell}_1$ is defined below.

\begin{theorem} 
For a complex manifold with an action of a finite group $G$, the following properties are equivalent:  $G$-BOPA, $G$-BOPI, $G$-BOPAJI, $G$-$\textrm{Ell}_1$, and the $G$-Oka property.
\end{theorem}

We believe that generalising Theorem \ref{mainthm:manifolds} to arbitrary reductive groups will require new methods.  In \cite[Section 5]{KLS-2021} we took the following step towards this goal.

\begin{theorem}
Let $G$ be a complex reductive group and $K$ a maximal compact subgroup of $G$.  Let $X$ be a Stein $G$-manifold and $Y$ be a $G$-Oka manifold.  Assume that $X$ has a single slice type, that is, the quotient map $X\to X\git G$ is a holomorphic $G$-fibre bundle.  Then every $K$-equivariant continuous map $X\to Y$ is homotopic, through such maps, to a $G$-equivariant holomorphic map.  
\end{theorem}

We now turn to the equivariant versions of two fundamental properties in Oka theory.  A manifold $Y$ is said to be elliptic -- Gromov's definition \cite{Gromov} marked the beginning of modern Oka theory -- if it carries a dominating spray, that is, there is a holomorphic map $s:E\to Y$, called a spray, defined on the total space of a holomorphic vector bundle $E$ on $Y$, such that $s(0_y)=y$ for all $y\in Y$, which is dominating in the sense that $s\vert_{E_y}:E_y\to Y$ is a submersion at $0_y$ for all $y\in Y$.  If a complex Lie group $G$ acts on $Y$, then we say that $s$ is a $G$-spray if the action on $Y$ lifts to an action on $E$ by vector bundle isomorphisms such that both $s$ and the projection $E\to Y$ are equivariant.  We say that $Y$ is \emph{$G$-elliptic} if it carries a dominating $G$-spray.

The weaker notion of relative $G$-ellipticity, also known as $G$-$\textrm{Ell}_1$ and mentioned above, is defined as follows.  The manifold $Y$ satisfies $G$-$\textrm{Ell}_1$ if for every holomorphic $G$-map $f$ from a Stein $G$-manifold $X$ to $Y$, there is a holomorphic $G$-vector bundle $E$ over $X$ and a dominating $G$-spray $s:E\to Y$ over $f$.  This means that $s(0_x)=f(x)$ for every $x\in X$, where $0_x$ is the zero vector in the fibre $E_x$ of $E$ over $x$, and $s\vert_{E_x}:E_x \to Y$ is a submersion at $0_x$.

We say that $Y$ is \emph{$G$-Runge} if for every Stein $G$-manifold $X$ and every $G$-invariant Runge domain $\Omega$ in $X$, the closure of the image of the restriction map $\O^G(X, Y) \to \O^G(\Omega, Y)$ is a union of path components (perhaps empty).  (To say that $\Omega$ is Runge means that $\Omega$ is Stein and the restriction map $\O(X)\to \O(\Omega)$ has dense image.)  In other words, approximability of holomorphic $G$-maps $\Omega\to Y$ by holomorphic $G$-maps $X\to Y$ is deformation-invariant.  When $G$ is the trivial group, the $G$-Runge property of $Y$ is one of the equivalent formulations of the Oka property.  When $G$ is reductive, the property of a domain $\Omega$ in $X$ being $G$-invariant and Runge can be described in several equivalent ways \cite[Section 6]{KLS-2021}.  For example, it is equivalent to say that $\Omega$ is the preimage of a Runge domain in $X\git G$.

Analogues of the basic results about the $G$-Oka property hold for both $G$-ellipticity and the $G$-Runge property \cite[Sections 3 and 6]{KLS-2021}, except we do not know a simple proof that a $G$-homogeneous space is $G$-Runge.  This is what we know about the relationships between the three properties.

\begin{theorem}  \label{thm:properties}
Let $G$ be a reductive complex group and $Y$ a $G$-manifold.

{\rm (a)}  If $Y$ is $G$-elliptic, then $Y$ is $G$-Runge.

{\rm (b)}  If $Y$ is $G$-Runge, then $Y$ is $G$-Oka.

{\rm (c)}  If $G$ is finite and $Y$ is Stein and $G$-Oka, then $Y$ is $G$-elliptic.
\end{theorem}

The proof of (a) is somewhat involved: it is an equivariant version of Gromov's linearisation method, sketched in \cite[Section 1.4]{Gromov}.  It uses some equivariant Stein theory, most importantly the equivariant version of Theorem B of Cartan and Serre, due to Roberts \cite{Roberts}.  The proof of (b) is a quick reduction to the fact that manifold satisfying the Runge property for trivial actions is Oka.  The proof of (c) is a straightforward adaptation of the well-known proof in the case of no action.

Now let $Y$ be a $G$-homogeneous space and take the trivial $G$-vector bundle $Y\times\mathfrak g\to Y$, where $G$ acts on its Lie algebra $\mathfrak g$ by the adjoint representation.  Then $Y\times\mathfrak g\to Y$, $(y,v)\mapsto \exp(v)\cdot y$, is a dominating $G$-spray, so $Y$ is $G$-elliptic.  By Theorem \ref{thm:properties}(a), $Y$ is $G$-Runge.

\section{Open problems}  
\label{sec:problems}

\begin{enumerate}
\item  Does the parametric version of Theorem \ref{mainthm:isomorphisms1} hold?  That is, in the setting of the theorem, is the inclusion of the space of $G$-biholomorphisms $X\to Y$ into the space of strict $G$-diffeomorphisms a weak homotopy equivalence with respect to the appropriate topology?  The same question for strong $G$-homeomorphisms.
\item  Is there a counterexample to Theorem \ref{mainthm:isomorphisms2} if $X$ does not have the infinitesimal lifting property?
\item Does Theorem \ref{thm:small-modules} hold if $G^0$ is a torus $(\C^*)^{n}$ of dimension $n\geq 2$?
\item  Build approximation and interpolation into Theorem \ref{mainthm:sections}.
\item  Show that the weak homotopy equivalence in Theorem \ref{mainthm:sections} is a genuine homotopy equivalence under suitable conditions.  There are, by now, several such results in the literature, the first in \cite{Larusson-2015}.  The same question for Theorem \ref{mainthm:principal}.
\item  Does Theorem \ref{mainthm:manifolds} hold for arbitrary actions of a reductive group $G$?  Does it hold for actions for which all the $G$-orbits are closed?
\item  Does the parametric version of Theorem \ref{mainthm:manifolds} hold?
\item  Let $G$ be a reductive group and $Y$ be a $G$-Oka manifold.  Is $Y$ $G$-Runge?
\end{enumerate}

\bibliographystyle{amsalpha}
\bibliography{Survey}

\providecommand{\bysame}{\leavevmode\hbox to3em{\hrulefill}\thinspace}
\providecommand{\MR}{\relax\ifhmode\unskip\space\fi MR }
\providecommand{\MRhref}[2]{%
  \href{http://www.ams.org/mathscinet-getitem?mr=#1}{#2}
}
\providecommand{\href}[2]{#2}
\begin{thebibliography}{KLS17b}

\bibitem[Car58]{Cartan}
Henri Cartan, \emph{Espaces fibr\'es analytiques}, Symposium internacional de
  topolog\'\i a algebraica, Universidad Nacional Aut\'onoma de M\'exico and
  UNESCO, Mexico City, 1958, pp.~97--121. \MR{0098196}

\bibitem[DK98]{Derksen-Kutzschebauch}
Harm Derksen and Frank Kutzschebauch, \emph{Nonlinearizable holomorphic group
  actions}, Math. Ann. \textbf{311} (1998), no.~1, 41--53.

\bibitem[For03]{Forstneric-2003}
Franc Forstneri\v{c}, \emph{The {O}ka principle for multivalued sections of
  ramified mappings}, Forum Math. \textbf{15} (2003), no.~2, 309--328.
  \MR{1956971}

\bibitem[For09]{Forstneric-2009}
\bysame, \emph{Oka manifolds}, C. R. Math. Acad. Sci. Paris \textbf{347}
  (2009), no.~17-18, 1017--1020. \MR{2554568}

\bibitem[For17]{Forstneric-book}
\bysame, \emph{Stein manifolds and holomorphic mappings}, second ed.,
  Ergebnisse der Mathematik und ihrer Grenzgebiete. 3. Folge. A Series of
  Modern Surveys in Mathematics, vol.~56, Springer, Cham, 2017, The homotopy
  principle in complex analysis. \MR{3700709}

\bibitem[FR66]{Forster-Ramspott}
Otto Forster and Karl~Josef Ramspott, \emph{Okasche {P}aare von {G}arben
  nicht-abelscher {G}ruppen}, Invent. Math. \textbf{1} (1966), 260--286.
  \MR{0212211}

\bibitem[Gra57a]{Grauert-Approximation}
Hans Grauert, \emph{Approximationss\"atze f\"ur holomorphe {F}unktionen mit
  {W}erten in komplexen {R}\"aumen}, Math. Ann. \textbf{133} (1957), 139--159.

\bibitem[Gra57b]{Grauert-Liesche}
\bysame, \emph{Holomorphe {F}unktionen mit {W}erten in komplexen {L}ieschen
  {G}ruppen}, Math. Ann. \textbf{133} (1957), 450--472. \MR{0098198}

\bibitem[Gra58]{Grauert-Faserungen}
\bysame, \emph{Analytische {F}aserungen \"uber holomorph-vollst\"andigen
  {R}\"aumen}, Math. Ann. \textbf{135} (1958), 263--273.

\bibitem[Gro89]{Gromov}
Mikhail Gromov, \emph{Oka's principle for holomorphic sections of elliptic
  bundles}, J. Amer. Math. Soc. \textbf{2} (1989), no.~4, 851--897.

\bibitem[Hei88]{Heinzner-1988}
Peter Heinzner, \emph{Linear \"aquivariante {E}inbettungen {S}teinscher
  {R}\"aume}, Math. Ann. \textbf{280} (1988), no.~1, 147--160.

\bibitem[Hei91]{Heinzner-1991}
\bysame, \emph{Geometric invariant theory on {S}tein spaces}, Math. Ann.
  \textbf{289} (1991), no.~4, 631--662. \MR{1103041}

\bibitem[HH94]{Heinzner-Huckleberry}
Peter Heinzner and Alan Huckleberry, \emph{Invariant plurisubharmonic
  exhaustions and retractions}, Manuscripta Math. \textbf{83} (1994), no.~1,
  19--29. \MR{1265915}

\bibitem[HK95]{Heinzner-Kutzschebauch}
Peter Heinzner and Frank Kutzschebauch, \emph{An equivariant version of
  {G}rauert's {O}ka principle}, Invent. Math. \textbf{119} (1995), no.~2,
  317--346. \MR{1312503}

\bibitem[Huc90]{Huckleberry}
Alan~T. Huckleberry, \emph{Actions of groups of holomorphic transformations},
  Several complex variables, {VI}, Encyclopaedia Math. Sci., vol.~69, Springer,
  Berlin, 1990, pp.~143--196.

\bibitem[KK08]{KK-2008}
Shulim Kaliman and Frank Kutzschebauch, \emph{Density property for
  hypersurfaces {$UV=P(\overline X)$}}, Math. Z. \textbf{258} (2008), no.~1,
  115--131. \MR{2350038}

\bibitem[KLS15]{KLS-2015}
Frank Kutzschebauch, Finnur L\'{a}russon, and Gerald~W. Schwarz, \emph{An {O}ka
  principle for equivariant isomorphisms}, J. Reine Angew. Math. \textbf{706}
  (2015), 193--214. \MR{3393367}

\bibitem[KLS17a]{KLS-2017}
\bysame, \emph{Homotopy principles for equivariant isomorphisms}, Trans. Amer.
  Math. Soc. \textbf{369} (2017), no.~10, 7251--7300. \MR{3683109}

\bibitem[KLS17b]{KLS-2017-lin}
\bysame, \emph{Sufficient conditions for holomorphic linearisation}, Transform.
  Groups \textbf{22} (2017), no.~2, 475--485. \MR{3649463}

\bibitem[KLS18]{KLS-2018}
\bysame, \emph{An equivariant parametric {O}ka principle for bundles of
  homogeneous spaces}, Math. Ann. \textbf{370} (2018), no.~1-2, 819--839.
  \MR{3747503}

\bibitem[KLS21]{KLS-2021}
\bysame, \emph{Gromov's {O}ka principle for equivariant maps}, J. Geom. Anal.
  \textbf{31} (2021), no.~6, 6102--6127. \MR{4267639}

\bibitem[Kra96]{Kraft-1996}
Hanspeter Kraft, \emph{Challenging problems on affine {$n$}-space},
  Ast\'erisque (1996), no.~237, Exp.\ No.\ 802, 5, 295--317, S{\'e}minaire
  Bourbaki, Vol. 1994/95.

\bibitem[KS21]{Kutzschebauch-Schwarz}
Frank Kutzschebauch and Gerald~W. Schwarz, \emph{A characterization of
  linearizability for holomorphic $\mathbb{C}^*$-actions}, 2021.

\bibitem[Kus21]{Kusakabe}
Yuta Kusakabe, \emph{Elliptic characterization and localization of {O}ka
  manifolds}, Indiana Univ. Math. J. \textbf{70} (2021), no.~3, 1039--1054.
  \MR{4284105}

\bibitem[L{\'a}r04]{Larusson-2004}
Finnur L{\'a}russon, \emph{Model structures and the {O}ka principle}, J. Pure
  Appl. Algebra \textbf{192} (2004), no.~1-3, 203--223. \MR{2067196}

\bibitem[L{\'a}r15]{Larusson-2015}
\bysame, \emph{Absolute neighbourhood retracts and spaces of holomorphic maps
  from {S}tein manifolds to {O}ka manifolds}, Proc. Amer. Math. Soc.
  \textbf{143} (2015), no.~3, 1159--1167. \MR{3293731}

\bibitem[Lun73]{Luna-1973}
Domingo Luna, \emph{Slices \'etales}, Sur les groupes alg\'ebriques, Soc. Math.
  France, Paris, 1973, pp.~81--105. Bull. Soc. Math. France, Paris, M\'emoire
  33.

\bibitem[Nee85]{Neeman-1985}
Amnon Neeman, \emph{The topology of quotient varieties}, Ann. of Math. (2)
  \textbf{122} (1985), no.~3, 419--459. \MR{819554}

\bibitem[Ram65]{Ramspott}
Karl-Josef Ramspott, \emph{Stetige und holomorphe {S}chnitte in {B}\"undeln mit
  homogener {F}aser}, Math. Z. \textbf{89} (1965), 234--246. \MR{0180701}

\bibitem[Rob86]{Roberts}
Mark Roberts, \emph{A note on coherent {$G$}-sheaves}, Math. Ann. \textbf{275}
  (1986), no.~4, 573--582.

\bibitem[Sch89]{Schwarz-1989b}
Gerald~W. Schwarz, \emph{Exotic algebraic group actions}, C. R. Acad. Sci.
  Paris S\'er. I Math. \textbf{309} (1989), no.~2, 89--94.

\bibitem[Sch95]{Schwarz-1995}
\bysame, \emph{Lifting differential operators from orbit spaces}, Ann. Sci.
  \'Ecole Norm. Sup. (4) \textbf{28} (1995), no.~3, 253--305.

\bibitem[Sch14]{Schwarz-2014}
\bysame, \emph{Quotients, automorphisms and differential operators}, J. Lond.
  Math. Soc. (2) \textbf{89} (2014), no.~1, 169--193.

\bibitem[Sch18]{Schwarz-2018}
\bysame, \emph{An {O}ka principle for {S}tein {$G$}-manifolds}, Indiana Univ.
  Math. J. \textbf{67} (2018), no.~5, 2045--2060. \MR{3875250}

\bibitem[Sno82]{Snow}
Dennis~M. Snow, \emph{Reductive group actions on {S}tein spaces}, Math. Ann.
  \textbf{259} (1982), no.~1, 79--97. \MR{656653}

\bibitem[Str68]{Strom}
Arne Str{\o}m, \emph{Note on cofibrations. {II}}, Math. Scand. \textbf{22}
  (1968), 130--142 (1969). \MR{0243525}

\end{thebibliography}

\end{document}